\documentclass[a4paper]{amsart}
\usepackage{amsmath,amssymb}

\newtheorem{theorem}{Theorem}[section]
\newtheorem{lemma}[theorem]{Lemma}
\newtheorem{corollary}[theorem]{Corollary}
\newtheorem{proposition}[theorem]{Proposition}

\theoremstyle{definition}

\newtheorem{remark}[theorem]{Remark}

\newcommand{\N}{\mathbb N}

\DeclareMathOperator{\ord}{ord}
\DeclareMathOperator{\lcm}{lcm}
\DeclareMathOperator{\supp}{supp}

\newcommand{\zsm}[2]{\mathsf{s}_{#1 \mathbb{N}}(#2)}

\renewcommand{\t}{\, | \,}
\renewcommand{\P}{\mathbb P}

\newcommand{\Sum}[2]{\underset{#1}{\overset{#2}{\sum}}}

\newcommand{\be}{\begin{equation}}
\newcommand{\ee}{\end{equation}}
\newcommand{\ber}{\begin{eqnarray}}
\newcommand{\eer}{\end{eqnarray}}
\newcommand{\nn}{\nonumber}
\newcommand{\ba}{\begin{align}}
\newcommand{\ea}{\end{align}}

\newcommand{\vp}{\mathsf{v}}

\newcommand{\und}{\;\mbox{ and } \;}

\numberwithin{equation}{section}

\begin{document}

\address{Institut f\"ur Mathematik und Wissenschaftliches Rechnen \\
Karl-Franzens-Universit\"at Graz \\
Heinrichstra\ss e 36\\
8010 Graz, Austria} \email{alfred.geroldinger@uni-graz.at,
diambri@hotmail.com}

\address{CMLS\\ Ecole polytechnique \\ 91128 Palaiseau
cedex, France} \email{wolfgang.schmid@math.polytechnique.fr}

\author[A. Geroldinger, D.J. Grynkiewicz, and W.A. Schmid]{Alfred Geroldinger and David J. Grynkiewicz and Wolfgang A. Schmid}

\thanks{This work was supported by
the Austrian Science Fund FWF, Project Numbers P21576-N18 and
J2907-N18}



\begin{abstract}
For a finite abelian group $G$ and a positive integer $d$, let
$\mathsf s_{d \mathbb N} (G)$ denote the smallest integer $\ell \in
\mathbb N_0$ such that every sequence $S$ over $G$ of length $|S| \ge
\ell$ has a nonempty zero-sum subsequence $T$ of length $|T| \equiv
0 \mod d$. We determine $\mathsf s_{d \mathbb N} (G)$
for all $d\geq 1$ when $G$  has rank at most two and, under mild conditions on $d$, also obtain precise values in the case of $p$-groups. In the same spirit, we obtain new upper
bounds for the Erd{\H o}s--Ginzburg--Ziv constant provided that, for
the $p$-subgroups $G_p$ of $G$, the Davenport constant $\mathsf D
(G_p)$ is bounded above by $2 \exp (G_p)-1$. This generalizes former
results for groups of rank two.
\end{abstract}

\title{Zero-sum problems with congruence conditions}

\maketitle

\bigskip
\section{Introduction}
\bigskip

Let $G$ be an additive finite abelian group. A direct zero-sum
problem, associated to a given Property P, asks for the extremal
conditions which guarantee  that every sequence $S$ over $G$
satisfying these conditions has a zero-sum subsequence with Property
P. Most of the properties studied so far deal with the length of the
zero-sum subsequence; others consider the cross number (see, e.g., \cite{GeKL}) or versions of this problem involving weights (see, e.g., \cite{AGS}).
In the case of lengths, a direct zero-sum problem asks for
the smallest integer $\ell \in \mathbb N_0$ such that every sequence
$S$ over $G$ of length $|S| \ge \ell$ has a zero-sum subsequence
of some prescribed length. This leads to the definition of the
following zero-sum constant:

\begin{quote}
For a subset $L \subset \N$, let $\mathsf{s}_L(G)$ denote the
smallest $\ell \in \N_0 \cup \{\infty\}$  such that every sequence \
$S$ \ over $G$ of length $|S| \ge \ell$ \ has a  zero-sum subsequence
$T$ of length $|T| \in L$.
\end{quote}
Note that $\mathsf{s}_L(G)=\infty$ if and only if $L \cap \exp(G)\mathbb{N}=\emptyset$.
The following sets lead to classical zero-sum invariants (the reader
may want to consult one of the surveys \cite{Ga-Ge06b, Ge09a} or the
monograph \cite{Ge-HK06a}):
\begin{itemize}
\item $\mathsf{s}_{\N}(G)=\mathsf{D}(G)$ is the Davenport constant,
\item $\mathsf{s}_{\{\exp(G)\}}(G)=\mathsf{s}(G)$ is the Erd{\H o}s--Ginzburg--Ziv constant,
\item $\mathsf{s}_{\{|G|\}}(G)=\mathsf{ZS}(G)$ is the  zero-sum constant, and
\item $\mathsf{s}_{[1,\exp(G)]}(G)=\eta(G)$ is the  $\eta$-invariant.
\end{itemize}

Moreover, $\mathsf{s}_L(G)$ has been investigated for various other sets, including:
$[1,k]$ for $k \ge \exp(G)$ (see, e.g., \cite{Delorme,Bh-SP07a, Fr-Sc10a}), $\{k \exp(G)\}$ for $k \in \N$ (see, e.g., \cite{GaoT,Kubertin}), \
 $\N \setminus k \N$ for $k \nmid \exp(G)$ and other unions of arithmetic progressions (see \cite{Restricted3,Sc01,Gi12a}), and $\exp(G)\N$ (see, e.g., \cite{Restricted4}). And, for recent closely related results, see, e.g., \cite{Gr05b,Ga-Ge07a,Ga-Pe09a,Gr10a,Ga-Ha-Wa10a,yuan}.

\smallskip
In the present paper, we investigate $\zsm{d}{G}$, first proving upper and lower bounds in terms
of a Davenport constant and its canonical lower bound. This allows us
to determine $\zsm{d}{G}$ for cyclic groups and, under
mild conditions on $d$, for $p$-groups (Theorem \ref{onsd_thm}).
Then we suppose that $d = \exp (G)$ and that, for the $p$-subgroups $G_p$ of
$G$, the Davenport constant $\mathsf D (G_p)$ is bounded above by $2
\exp (G_p)-1$ (note that every group of rank at most two satisfies
this condition). In this setting, we obtain  canonical upper bounds
for $\zsm{d}{G}$ and, among others, for the Erd{\H
o}s--Ginzburg--Ziv constant $\mathsf s (G)$ (Theorem
\ref{ind_thm_bound}, and see Theorem \ref{ind_thm_equ} for
a result in a similar vein). Next, using a more involved argument, we determine $\zsm{d}{G}$ for rank $2$ groups $G$, showing that $\zsm{d}{G}$ attains the value that would easily follow from our bounds if the conjectured value of $\mathsf D(G)$ for rank $3$ groups were true.
In the final section, we
apply these results to a problem from the theory of non-unique
factorizations which motivated the present investigations.

\bigskip
\centerline{\it Throughout this paper, let $G$ be a finite abelian
group.}

\bigskip
\section{Preliminaries} \label{2}
\bigskip

Our notation and terminology are consistent with \cite{Ga-Ge-Sc07a}
and \cite{Ge-HK06a}. We briefly gather some key notions and fix the
notation concerning sequences over  abelian groups. Let $\mathbb N$
denote the set of positive integers, let $\mathbb P \subset \mathbb N$ be
the set of prime numbers and let $\mathbb N_0 = \mathbb N \cup \{ 0
\}$. For real numbers $a, b \in \mathbb R$, we set $[a, b] = \{ x
\in \mathbb Z \mid a \le x \le b\}$. For $n \in \mathbb N$ and $p
\in \P$, let $C_n$ denote a cyclic group with $n$ elements
 and $\mathsf v_p (n) \in \N_0$ the
$p$-adic valuation of $n$ with $\mathsf v_p (p) = 1$. Throughout,
all abelian groups will be written additively.

For a subset $G_0 \subset G$, let $\langle G_0 \rangle $ denote the
subgroup generated by $G_0$. For a prime $p \in \P$, we denote by
$G_p = \{g \in G \mid \ord (g) \ \text{is a power of} \ p \}$ the
$p$-primary component of $G$. Suppose that $G \cong C_{n_1} \oplus
\ldots \oplus C_{n_r}$ with $r \in \mathbb N_0$ and $1 < n_1 \t
\ldots \t n_r$. Then $r = \mathsf r (G)$ will be called the {\it
rank} of $G$, and we set
\[
\mathsf D^* (G) = 1 + \sum_{i=1}^r (n_i-1) \,.
\]
Note that $\mathsf r (G) = 0$ and $\mathsf D^* (G) = 1$ for $G$
trivial.

Let $\mathcal F(G)$ be the free abelian monoid with basis $G$. The
elements of $\mathcal F(G)$ are called \ {\it sequences} \ over $G$.
We write sequences $S \in \mathcal F (G)$ in the form
\[
S =  \prod_{g \in G} g^{\mathsf v_g (S)}\,, \quad \text{with} \quad
\mathsf v_g (S) \in \mathbb N_0 \quad \text{for all} \quad g \in G
\,.
\]
We call \ $\mathsf v_g (S)$  the \ {\it multiplicity} \ of $g$ in
$S$, and we say that $S$ \ {\it contains} \ $g$ \ if \ $\mathsf v_g
(S) > 0$.   A sequence $S_1 $ is called a \ {\it subsequence} \ of
$S$ \ if \ $S_1 \, | \, S$ \ in $\mathcal F (G)$ \ (equivalently, \
$\mathsf v_g (S_1) \le \mathsf v_g (S)$ \ for all $g \in G$). If a
sequence $S \in \mathcal F(G)$ is written in the form $S = g_1 \cdot
\ldots \cdot g_l$, we tacitly assume that $l \in \mathbb N_0$ and
$g_1, \ldots, g_l \in G$.

\smallskip

For a sequence
\[
S \ = \ g_1 \cdot \ldots \cdot g_l \ = \  \prod_{g \in G} g^{\mathsf
v_g (S)} \ \in \mathcal F(G) \,,
\]
we call
\[
|S| = l = \sum_{g \in G} \mathsf v_g (S) \in \mathbb N_0 \qquad
\text{the \ {\it length} \ of \ $S$}\,,
\]
\[
\supp (S) = \{g \in G \mid \mathsf v_g (S) > 0 \} \subset G \qquad
\text{the \ {\it support} \ of \ $S$} \quad \text{and}
\]
\[
\sigma (S) = \sum_{i = 1}^l g_i = \sum_{g \in G} \mathsf v_g (S) g
\in G \qquad \text{the \ {\it sum} \ of \ $S$}\,.
\]
The sequence \ $S$ \ is called
\begin{itemize}

\item a {\it zero-sum sequence} \ if \ $\sigma (S) = 0$,

\item {\it zero-sum free} \ if there is no  nonempty zero-sum subsequence,

\item a {\it minimal zero-sum sequence} \ if $S$ is a nonempty zero-sum sequence and every $S'|S$ with $1\leq |S'|<|S|$  is
      zero-sum free.
\end{itemize}
Every map of abelian groups \ $\varphi \colon G \to H$ \ extends to
a homomorphism \ $\varphi \colon \mathcal F (G) \to \mathcal F (H)$
\ where \ $\varphi (S) = \varphi (g_1) \cdot \ldots \cdot \varphi
(g_l)$.  If $\varphi$ is a homomorphism, then $\varphi (S)$ is a
zero-sum sequence if and only if $\sigma (S) \in \text{\rm Ker}
(\varphi)$. We let $\mathcal A(G)$ denote the set of all minimal zero-sum sequences over $G$.

\bigskip
\section{Basic bounds and results for cyclic and $p$-groups}
\bigskip

In this section, we establish some of our results on $\zsm{d}{G}$.
In particular, we obtain the following result.

\begin{theorem} \label{onsd_thm}
Let $d \in \N$ and let $n= \exp(G)$.
\begin{enumerate}
\item Suppose $G$ is cyclic. Then $\zsm{d}{G}= \mathsf D^*(G\oplus C_d)=\lcm(n,d)+ \gcd(n,d) - 1$.
\item Suppose $G$ is $p$-group.
\begin{enumerate}
\item  $\zsm{d}{G} = \mathsf D^*(G\oplus C_d)=\mathsf{D}^{\ast} (G) + d -1$ for $d=p^{\alpha}$ with $\alpha \in \N_0$.
\item $\zsm{d}{G} = \mathsf D^*(G\oplus C_d)=\mathsf{D}^{\ast}(G) + d-1$ for each $d \in \N$ with
$\mathsf{D}^{\ast}(G)\le p^{\mathsf{v}_p(d)}$.
\item $\zsm{d}{G} =\mathsf D^*(G\oplus C_d)= \mathsf{D}^{\ast}(G)-n + \lcm(n,d)+ \gcd(n,d) - 1$ for each $d \in \N$ with $p^{\mathsf{v}_p(d)} \le 2 n - \mathsf{D}^{\ast}(G)$.
\end{enumerate}
\end{enumerate}
\end{theorem}

The strategy to prove this result is to bound $\zsm{d}{G}$, for
generic $d$ and $G$, in terms of the invariants
$\mathsf{D}^{\ast}(\cdot)$ and $\mathsf{D}(\cdot)$, and then to make
these `abstract bounds' explicit invoking results on the Davenport
constant. We remark that Theorem \ref{onsd_thm}.2(a) for $\alpha=1$ can also be derived
as a special case of \cite[Theorem 2.3]{Gi12a}, proved via the Combinatorial Nullstellensatz.
The first part is carried out in
Proposition \ref{onsd_prop_bounds}. However, since the lower bound given there is in terms of $\mathsf D^*(G\oplus C_d)$, we begin first with the following lemma showing how to calculate $\mathsf D^*(G\oplus C_d)$ explicitly.

\medskip
\begin{lemma}\label{D*lemma} Let $d \in \N$ and let $G\cong C_{n_1}\oplus\cdots\oplus C_{n_r}$ with $1<n_1\mid \cdots\mid n_r$. Set  $n_0=1$ and $n_{r+1}=0$. Then $G\oplus C_d\cong C_{m_0}\oplus\cdots\oplus C_{m_r}$ with $1\leq m_0\mid\cdots\mid m_r$, so that $$\mathsf D^*(G\oplus C_d)=\Sum{i=0}{r}(m_i-1)+1,$$ where $$m_i=n_i\frac{\gcd(n_{i+1},d)}{\gcd(n_{i},d)}=\gcd(n_{i+1},\lcm(n_i,d))=\lcm(n_i,\gcd(n_{i+1},d))\quad \mbox{ for } i\in [0,r].$$
\end{lemma}

\begin{proof}
Letting $p_1,\ldots,p_k$ be the distinct prime divisors $p_i$ of $n_r$, we have $$G\cong \bigoplus_{i=1}^k\bigoplus_{j=1}^rC_{p_i^{s_{i,j}}},$$ where $0\leq s_{i,1}\leq\ldots\leq s_{i,r}$ for all $i\in [1,k]$ and $p_1^{s_{1,j}}\cdot\ldots\cdot p_k^{s_{k,j}}=n_j$ for all $j\in [1,r]$. Let $$C_d\cong C_m\oplus \bigoplus_{i=1}^{k}C_{p_i^{s'_i}},$$ where $\mathsf{v}_{p_i}(d)=s'_i\ge 0$ and $d=mp_1^{s'_1}\cdots p_k^{s'_k}$. Recall $n_0=1$,
$n_{r+1}=0$ and write \be\label{ants}G\oplus C_d\cong
C_{m_0}\oplus\ldots \oplus C_{m_r}\ee with $1\leq  m_0 \mid \dots \mid m_r$.
Then \be\label{wooster}\vp_{p_i}(m_j)=\left\{
                                                                  \begin{array}{ll}
                                                                    \vp_{p_i}(n_j)=s_{i,j}, & \hbox{if } \vp_{p_i}(d)\leq \vp_{p_i}(n_j),\\
 \vp_{p_i}(d)=s'_{i}, & \hbox{if }\vp_{p_i}(n_j)\leq \vp_{p_i}(d)\leq \vp_{p_i}(n_{j+1}),\\
\vp_{p_i}(n_{j+1})=s_{i,j+1}, & \hbox{if }   \vp_{p_i}(n_{j+1})\leq \vp_{p_i}(d),                                                                \end{array}
                                                                \right.\ee
for $j\in [0,r]$ and $i\in [1,k]$.

We claim that \be\label{booster}m_j=n_j\frac{\gcd(n_{j+1},d)}{\gcd(n_{j},d)},\ee  for $j\in [0,r]$. Indeed, since $m_r=\exp(G\oplus C_d)=n_r\frac{d}{\gcd(n_r,d)}$, this is clear for $j=r$. To see this also holds for $j<r$, it suffices to see  $\vp_{p_i}\left(n_j\frac{\gcd(n_{j+1},d)}{\gcd(n_{j},d)}\right)$ agrees with \eqref{wooster} for each $i\in [1,k]$. However, $$\vp_{p_i}\left(n_j\frac{\gcd(n_{j+1},d)}{\gcd(n_{j},d)}\right)=\vp_{p_i}(n_j)+\min\{\vp_{p_i}(d),\,\vp_{p_i}(n_{j+1})\}-
\min\{\vp_{p_i}(d),\,\vp_{p_i}(n_j)\},$$ which is easily seen to agree with \eqref{wooster} in all three cases, completing the claim. Thus from \eqref{ants} we conclude that $$\mathsf D^*(G\oplus C_d)=\Sum{i=0}{r}(m_i-1)+1.$$ Moreover, in view of \eqref{booster}, one sees that the expression for the $m_i$ can be rewritten as $$m_i=\gcd(n_{i+1},\lcm(n_i,d))=\lcm(n_i,\gcd(n_{i+1},d))\quad \mbox{ for } i\in [0,r],$$ which completes the proof.
\end{proof}

\medskip
\begin{proposition}
\label{onsd_prop_bounds}
Let $d \in \N$.
Then $$\mathsf D^*(G)+d-1\leq \mathsf D^*(G\oplus C_d)\leq \mathsf s_{d\N}(G)\leq \mathsf D(G\oplus C_d)$$ and $$\mathsf D(G)+d-1\leq  \mathsf s_{d\N}(G).$$
In particular, if $\mathsf{D}^{\ast}(G\oplus C_d) = \mathsf{D}(G\oplus C_d)$, then  $\zsm{d}{G} = \mathsf{D}^{\ast}(G\oplus C_d)$.
\end{proposition}

\begin{proof}
Write $G\cong C_{n_1}\oplus \ldots\oplus C_{n_r}$, where $1< n_1 \mid \dots \mid n_r$, let $n_0=1$ and $n_{r+1}=0$, and let $e_1, \dots , e_r \in G$ be such that $G =
\oplus_{i=1}^r \langle e_i \rangle$ with $ \ord(e_i)=n_i$. By Lemma \ref{D*lemma}, we know
\be\label{DDD}\mathsf D^*(G\oplus C_d)=\Sum{i=0}{r}(m_i-1)+1,\ee where $$m_i=n_i\frac{\gcd(n_{i+1},d)}{\gcd(n_{i},d)}=\gcd(n_{i+1},\lcm(n_i,d))=\lcm(n_i,\gcd(n_{i+1},d))\quad \mbox{ for } i\in [0,r].$$

We begin by showing $\mathsf D^*(G)+d-1\leq \mathsf D^*(G\oplus C_d)$.
From \eqref{DDD}, we have $$\mathsf D^*(G\oplus C_d)=\Sum{i=0}{r}m_i-r=\Sum{i=0}{r}d_in_i-r,$$ where \be\label{bees}d_i=\frac{\gcd(n_{i+1},d)}{\gcd(n_{i},d)}\quad\mbox{ for }i\in [0,r].\ee Observing that $d_0\cdots d_r=d$ with $d_i\in [1,d]$ for all $i$, and noting that $d_0n_0=d_0=\gcd(n_1,d)\mid n_j$ for all $j$, so that $d_0n_0\leq n_j$, it is easily verified that the above expression is minimized when $d_0=d$ and $d_i=1$ for $i\geq 1$, in which case  $\mathsf D^*(G\oplus C_d)\geq d+\Sum{i=1}{r}n_i-r=\mathsf D^*(G)+d-1$, as desired.

Next, we show $\mathsf D(G)+d-1\leq  \mathsf s_{d\N}(G).$
By definition of $\mathsf{D}(G)$, there exists a zero-sum free sequence $S \in \mathcal{F}(G)$ with $|S|= \mathsf{D}(G)-1$.
We consider the sequence $0^{d-1}S$. Clearly, the only nonempty zero-sum subsequences of $0^{d-1}S$
are the sequences $0^k$ with $k \in [1,d-1]$.
Thus $\zsm{d}{G} > |0^{d-1}S|= \mathsf{D}(G)+d-2$, establishing our claim.

We proceed to show the remaining lower bound $\mathsf D^*(G\oplus C_d)\leq \mathsf s_{d\N}(G)$.
Let  $e_0=0$ and let $S=\prod_{i=0}^r e_i^{m_i-1}$. From \eqref{DDD}, we have $|S|=\mathsf{D}^{\ast}(G \oplus C_d)-1$.
Consider  $T\mid S$ with $\sigma(T)=0$ and $d\mid |T|$.
We will show that $|T|=0$, establishing the lower bound.
Let $v_i = \mathsf{v}_{e_i}(T)$ for $i\in [0,r]$. We note that $n_i=\ord(e_i) \mid v_i$ for each $i$, and we set $x_i= v_i/n_i$.
By the very definition, we have
\[|T|= \sum_{i=0}^r x_i n_i.\]
Note that $v_i=x_in_i< m_i$ (as $\vp_{e_i}(S)<m_i$ with $T\mid S$), and thus $x_i\in [0,d_i -1]$ for each $i$.
We have to show that $x_i=0$ for each $i$. Assume not, and let $j\in [0,r]$ be minimal with $x_j \neq 0$.
Since \[|T|=\Sum{i=1}{r}x_in_i=\sum_{i=j}^r x_i n_i\] is divisible by $d$,
we get that (for $j=r$, the right-hand side below is $0$)
\[x_jn_j \equiv -n_{j+1}\sum_{i=j+1}^r x_i \frac{n_i}{n_{j+1}} \pmod{d},\]
and thus $\gcd(n_{j+1},d)\mid x_jn_j$.
Consequently, $\frac{\gcd(n_{j+1},d)}{\gcd(n_j,d)}\mid \frac{x_jn_j}{\gcd(n_j,d)}$, whence \eqref{bees} implies
\[d_j \ \big \vert \ x_j \frac{n_j}{\gcd(n_j,d)}.\]
Noting from \eqref{bees} that
\[\gcd(d_j,  \frac{n_j}{\gcd(n_j,d)})=1,\]
it follows that $d_j \mid x_j$, which in view of $x_j \in [0, d_j - 1]$ implies $x_j=0$.
This contradicts  the definition of $x_j$ and completes the argument.

It remains to show  $\mathsf s_{d\N}(G)\leq \mathsf D(G\oplus C_d)$. Let $S \in \mathcal{F}(G)$ with $|S|\ge \mathsf{D}(G
\oplus C_d)$. We have to show that $S$ has a nonempty zero-sum
subsequence of length congruent to $0$ modulo $d$. Let $e\in G
\oplus C_d$ be such that $G \oplus C_d = G \oplus \langle e \rangle$,
and let $\iota \colon G \to G \oplus C_d$ denote the map defined via
$\iota(g)=g +e$.
Since $|\iota(S)|=|S|\geq \mathsf D(G\oplus C_d)$, applying the definition of $\mathsf D(G\oplus C_d)$ to $\iota(S)$ yields
a nonempty subsequence $T\mid S$ with $0=\sigma(\iota(T))=\sigma(T)+|T|e$. Hence $T$ is a zero-sum subsequence with length
$|T|$ divisible by
$\ord(e)=n$, as desired.
\end{proof}

Now we prove Theorem \ref{onsd_thm}. We need the following
well-known results on the Davenport constant, which will be used later in the paper as well. Namely,
$\mathsf{D}(G)=\mathsf{D}^{\ast}(G)$ if $G$ satisfies any one
of the following conditions (see \cite{Ge09a}, specifically Theorems 2.2.6 and 4.2.10 and
Corollary 4.2.13):
\begin{itemize}
\item $G$ has rank at most two.
\item $G$ is a $p$-group.
\item $G\cong G' \oplus C_n$ where $G'$ is a $p$-group with $\mathsf{D}^{\ast}(G')\le 2 \exp(G')-1$ and $p\nmid n$.
\end{itemize}

\medskip
\begin{proof}[Proof of Theorem \ref{onsd_thm}]
1. As $G$ is cyclic, it follows from Lemma \ref{D*lemma} that $G\oplus C_d \cong C_{\gcd(n,d)} \oplus C_{\lcm(n,d)}$.
Thus, by the above mentioned results, we know that $\mathsf{D}(C_{\gcd(n,d)} \oplus C_{\lcm(n,d)})=\mathsf{D}^{\ast}(C_{\gcd(n,d)} \oplus C_{\lcm(n,d)})=
\gcd(n,d)+\lcm(n,d)-1$, whence Proposition \ref{onsd_prop_bounds} completes the proof of part 1.

\smallskip
\noindent
2. Let $H$ be a group such that $G\cong H \oplus C_{n}$. As $G$ is a $p$-group, it follows from Lemma \ref{D*lemma} that  $$G \oplus C_{d}\cong H \oplus C_{\gcd(n,d)} \oplus C_{\lcm(n,d)}$$ with $\lcm(n,d)$ the exponent of $G \oplus C_{d}$. Consequently, since $G$ is a $p$-group, it follows that $\mathsf{D}^{\ast} (G \oplus C_d) = \mathsf{D}^{\ast}(H) + \gcd(n,d) + \lcm(n,d) - 2$.
Observe that this quantity is equal to the value we claim for $\zsm{d}{G}$ in each of the points (a), (b), and (c), with this being the case in  (b) since $p^{\vp_p(d)}\geq \mathsf D^*(G)\geq n$ with $n$ being a power of $p$ (as $G$ is a $p$-group) implies $\lcm(n,d)=d$ and $\gcd(n,d)=n$.
Thus, again, by Proposition \ref{onsd_thm} it suffices to show that $\mathsf{D} (G \oplus C_d) =  \mathsf{D}^{\ast} (G \oplus C_d)$.
For (a), \ $G\oplus C_d$ is a $p$-group and the claim is immediate by the above mentioned result for $p$-groups.
For (b) and (c), let $\alpha_1 \in \N_0$ be such that $\gcd(n,d)=p^{\alpha_1}$ and let $\alpha_2= \mathsf{v}_p(\lcm(n,d))$.

Suppose the hypotheses of (b) hold. Then $n\leq \mathsf D^*(G)\leq p^{\vp_p(d)}$ so that $p^{\alpha_2}=p^{\vp_p(d)}$ and $p^{\alpha_1}=n$. Hence, using the hypothesis $n\leq \mathsf D^*(G)\leq p^{\vp_p(d)}=p^{\alpha_2}$ once more, we find that $$\mathsf D(H\oplus C_{p^{\alpha_1}} \oplus C_{p^{\alpha_2}})=\mathsf D^*(H\oplus C_{p^{\alpha_1}} \oplus C_{p^{\alpha_2}})=\mathsf D^*(G)+p^{\alpha_2}-1 \leq p^{\vp_p(d)}+p^{\alpha_2}-1=2p^{\alpha_2}-1.$$ Thus the $p$-group $H\oplus C_{p^{\alpha_1}} \oplus C_{p^{\alpha_2}}$ fulfils the conditions imposed in the last of the above mentioned results,
completing the proof of (b).

Suppose the hypotheses of (c) hold. If $p^{\alpha_2}=p^{\vp_p(d)}$, then $n\leq p^{\vp_p(d)}$, whence the hypothesis of (c) implies $\mathsf D^*(G)\leq n$. As a result, since $n\leq \mathsf D^*(G)$ with equality if and only if $G$ is cyclic, we conclude that $G$ is cyclic. Consequently, $G\oplus C_d$ has rank at most $2$, so that $\mathsf D^*(G\oplus C_d)=\mathsf D(G\oplus C_d)$ by the first of the above mentioned results, and now the result follows from Proposition \ref{ind_thm_bound}. Therefore it remains to consider the case when $p^{\alpha_2}=n$ and $p^{\alpha_1}=p^{\vp_p(d)}$. In this case, the hypothesis of (c) instead implies$$\mathsf D(H\oplus C_{p^{\alpha_1}} \oplus C_{p^{\alpha_2}})=\mathsf D^*(H\oplus C_{p^{\alpha_1}} \oplus C_{p^{\alpha_2}})=\mathsf D^*(G)+p^{\alpha_1}-1 \leq 2n-1=2p^{\alpha_2}-1.$$ Thus the $p$-group $H\oplus C_{p^{\alpha_1}} \oplus C_{p^{\alpha_2}}$ fulfils the conditions imposed in the last of the above mentioned results,
completing the proof of (c).
\end{proof}

Several results on the
Davenport constant, in addition to those already recalled,  are known (see, e.g., \cite{Ga-Ge06b} for an
overview). Essentially, each of them allows one to obtain some additional
insight on $\zsm{d}{G}$ via Proposition \ref{onsd_prop_bounds}. For example,
it is conjectured that $\mathsf{D}^{\ast}(G)= \mathsf{D}(G)$ for
groups of rank three (see \cite[Conjecture 3.5]{Ga-Ge06b}; and
\cite{Bh-SP07a} and \cite{Sc10c} for recent results, confirming this
conjecture in special cases). If this were the case, then, for groups
of rank two, $\zsm{d}{G}=\mathsf D^*(G\oplus C_d)$ would immediately follow from Proposition \ref{onsd_prop_bounds} for all $d\in \N$. In Section \ref{sec-rank2}, we will show this equality holds without the use of the conjectured value of $\mathsf D(G)$ for rank three groups, which could be construed as giving weak evidence for the supposed value.

Of course, the two invariants $\mathsf{D}(\cdot)$ and
$\mathsf{D}^{\ast}(\cdot)$  are not equal for all finite abelian
groups, and there are examples of pairs $(d,G)$ for which the bounds
in Proposition \ref{onsd_prop_bounds} do not coincide, i.e., $$\max
\{\mathsf{D}(G) + d - 1, \mathsf{D}^{\ast}(G\oplus C_d)\}<
\mathsf{D}(G\oplus C_d);$$ see \cite{Ge-Sc92, Ge-Li-Ph11} for more
information on the phenomenon of inequality of $\mathsf{D}(\cdot)$
and $\mathsf{D}^{\ast}(\cdot)$. However, it is conjectured that the
difference between $\mathsf{D}(G)$ and $\mathsf{D}^{\ast}(G)$ is
fairly small for any $G$ (in a relative sense)---indeed, there is a
conjecture that asserts that this difference is at most
$\mathsf{r}(G)-1$ (see \cite[Conjecture 3.7]{Ga-Ge06b})---and thus
the combination of the bounds of Proposition \ref{onsd_prop_bounds} would in
general yield a good approximation for $\zsm{d}{G}$.

\bigskip
\section{Results when $\mathsf{D}(G_p) \le 2 \exp(G_p) - 1$}
\label{sec_ind}
\bigskip

We use the inductive method to obtain upper bounds on
$\mathsf{D}(G)$, $\mathsf{s}(G)$, $\eta(G)$ and $\zsm{d}{G}$,
imposing conditions on the $p$-subgroups of $G$. These conditions
are fulfilled, in particular, for groups of rank at most two. Recall
that the question of whether or not $\mathsf s (C_p \oplus C_p) \le
4p-3$ holds for all primes $p \in \mathbb P$ was open for more then
20 years (the Kemnitz Conjecture), and finally solved by C. Reiher
\cite{Re07a}. His result was then generalized to arbitrary groups of
rank two \cite[Theorem 5.8.3]{Ge-HK06a}, and to $p$-groups $G$ with
$\mathsf D (G) \le 2 \exp (G)-1$ \cite[Theorem 1.2]{Sc-Zh09a}. We
refer to \cite[Section 4]{Ge09a} for a survey on the Erd{\H
o}s--Ginzburg--Ziv constant, and to \cite{Ka-Pa-Sa09a, Ka-Pa-Sa10a}
for some recent connections. The upper bound for
$\mathsf{s}_{n\mathbb{N}}(G)$ for groups of rank two was first given
in \cite[Theorem 6.7]{Ga-Ge06b}. Note that the upper bound
$\mathsf{\eta}(G) \le 3n-2$ is precisely what is needed in various
applications (see for example \cite{Ge-HK06a}).

\medskip
\begin{theorem}
\label{ind_thm_bound} Let  $\exp(G)=n$.  Suppose that, for each $p
\in \P$, we have $\mathsf{D}(G_p) \le 2\exp(G_p)-1$.
\begin{enumerate}
\item The following inequalities hold\,{\rm :}
      \begin{enumerate}
      \item $\mathsf{D}(G) \le 2n-1$.
      \item $\mathsf{s}_{n\mathbb{N}}(G)\le 3 n - 2$.
            \end{enumerate}

\smallskip
\item If $\exp(G)$ is odd, then the following inequalities hold\,{\rm :}
      \begin{enumerate}
      \item $\mathsf{\eta}(G) \le 3n-2$.
      \item $\mathsf{s}(G) \le 4n-3$.
      \end{enumerate}
\end{enumerate}
\end{theorem}

In some cases, we are even able to establish the exact value of
these constants, though we have to impose more restrictive conditions. We do not include $\mathsf{D}(G)$ in the result
below since, in this case, an assertion of this form is well-known
(see the result mentioned before the proof of Theorem
\ref{onsd_thm}).

\begin{theorem}
\label{ind_thm_equ} Let  $\exp(G)=n$. Suppose there exists some odd
$q \in \P$ such that $\mathsf{D}(G_q)-  \exp(G_q)+1 \mid \exp(G_q)$
and $G_p$ is cyclic for each $p \in \P \setminus \{q\}$.
\begin{enumerate}
\item $\mathsf{\eta}(G) = 2(\mathsf{D}(G_q)-\exp(G_q)) + n$.

\item $\mathsf{s}(G) = 2(\mathsf{D}(G_q)-\exp(G_q)) + 2n - 1$.

\item \(\mathsf{s}_{d \mathbb{N}}(G)= \mathsf{D}(G_q)-  \exp(G_q) + \gcd(n,d)+\lcm(n,d)-1\)
      for each $d \in \N$ with  $(\mathsf{D}(G_q) - \exp(G_q)+1) \mid d$.
\end{enumerate}
\end{theorem}

\medskip

For both of the proofs below, we use \cite[Proposition 5.7.11]{Ge-HK06a}, which states that if $K\leq G$ and $\exp(G)=\exp(K)\exp(G/K)$, then \ber\label{inductive-exp-bounds-eta} \eta(G)&\leq & \exp(G/K)(\eta(K)-1)+\eta(G/K)\und\\
\label{inductive-exp-bounds-s}\mathsf s(G)&\leq &
\exp(G/K)(\mathsf s(K)-1)+\mathsf s(G/K).\eer
\medskip

\begin{proof}[Proof of Theorem \ref{ind_thm_bound}]
Let $p_1, \dots, p_s$ be the distinct primes such that $G= G_{p_1} \oplus \dots \oplus G_{p_s}$ is the decomposition of $G$ into non-trivial $p$-groups.

First, we establish the claims on $\eta(G)$ and $\mathsf{s}(G)$. Thus, we (temporarily) assume that each $p_i$ is odd.
We induct on $s$. For $s=0$, the claim is
trivial, and for $s=1$, it is an immediate consequence of
\cite[Theorem 1.2]{Sc-Zh09a}, which asserts that, for $H$ a $p$-group with $p$ an odd prime and $\mathsf{D}(H)\le 2\exp(H)-1$, one has $\mathsf{\eta}(H)\le  \mathsf{D}(H) +  \exp(H)-1$ and $\mathsf{s}(H)\le \mathsf{D}(H) + 2 \exp(H)-2$.

Suppose $s \ge 2$ and the claims hold true for $s-1$.  Since
$\exp(G)=\exp(G/G_{p_s})\exp(G_{p_s})$, we can invoke \eqref{inductive-exp-bounds-eta} and \eqref{inductive-exp-bounds-s} to conclude
\ber\label{teefel}\eta(G)&\le& \exp(G/G_{p_s})(\eta(G_{p_s})-1) + \eta(G/G_{p_s})\und\\
\label{eengel}\mathsf{s}(G)&\le& \exp(G/G_{p_s})(\mathsf{s}(G_{p_s})-1) + \mathsf{s}(G/G_{p_s}).\eer
By induction hypothesis, we have
\ber\nn\eta(G/G_{p_s})\leq 3\exp(G)/\exp(G_{p_s})-2,\quad \eta(G_{p_s})\leq 3\exp(G_{p_s})-2,\\\nn \mathsf{s}(G/G_{p_s})\leq 4\exp(G)/\exp(G_{p_s})-3,\und \quad \mathsf{s}(G_{p_s})\leq 4\exp(G_{p_s})-3.\eer Combining these inequalities with \eqref{teefel} and \eqref{eengel}
yields the desired bounds.

Next, we prove the result on $\mathsf s_{n \N} (G)$ and $\mathsf{D}(G)$.  However, the upper bound on $\mathsf{D}(G)$ follows from Proposition \ref{onsd_prop_bounds} and part 1(b), so it suffices to show 1(b). To do so, we  drop the assumption that each $p_i$ is odd. Of course, at most one of the $p_i$'s is even, and thus we may assume that $p_1, \dots , p_{s-1}$ are odd.
Again, we induct on $s$. The case $s=0$ is trivial. If  $s=1$, then $G=G_{p_1}$ is a $p_1$-group, so that $\mathsf D(G)=\mathsf D^*(G)$ by the previously mentioned results on the Davenport constant, in which case  Proposition \ref{onsd_prop_bounds} and our hypotheses, keeping in mind that $n=\exp(G)=\exp(G_{p_1})$, imply $$\mathsf s_{n\N}(G)=\mathsf D^*(G\oplus C_n)=\mathsf D^*(G_{p_1}\oplus C_n)=\mathsf D^*(G_{p_1})+n-1\leq 2\exp(G_{p_1})-1+n-1=3n-2,$$ as desired. This completes the base of the induction.

Suppose $s \ge 2$ and the claim holds true for $s-1$.
Let $\varphi \colon G \to G/G_{p_s} \cong G_{p_1} \oplus \dots
\oplus G_{p_{s-1}}$ denote the canonical epimorphism. Let $S \in
\mathcal{F}(G)$ with $|S| \ge 3n - 2$. Let $m= \exp(G_{p_s})$. Since
$$|S|\geq 3n-2= (3m-4)n/m + 4n/m -2$$ and since
$\mathsf{s}(G/G_{p_s})\le 4n/m -3$ holds by Theorem \ref{ind_thm_bound}.2(b), it follows that $S$ admits a product
decomposition $S = S_1 \cdot \ldots \cdot S_{3m-3}S'$ such that each $\varphi (S_i)$ has sum zero and length $|S_i| = n/m$,
where
$S_1,
\ldots, S_{3m-3}, S' \in \mathcal F (G)$ (see
\cite[Lemma 5.7.10]{Ge-HK06a}).

In view of  $|S'| \geq 3n-2-(3m-3)\frac{n}{m}=3\frac{n}{m}-2$ and the induction hypothesis, $S'$ has a
subsequence $S_{3m-2}$ such that $n/m \mid |S_{3m-2}|$ and
$\sigma(\varphi(S_{3m-2}))=0$. Now, for some generating element $e
\in C_n$, let $\iota \colon G \to G\oplus C_n$ denote the map
defined via $\iota(g)= g +e$. Then $\sigma(\iota(T))=
\sigma(T)+|T|e $ for each $T \in \mathcal{F}(G)$; in particular,
$\sigma(\iota(S_i)) \in G_{p_s} \oplus \langle (n/m) e \rangle$ for
each $i \in [1,3m-2]$. Since $$\mathsf{D}(G_{p_s} \oplus \langle
(n/m) e \rangle) = \mathsf{D}(G_{p_s})+ m -1 \le 3m-2,$$ it follows
that the sequence $\prod_{i=1}^{3m-2} \sigma(\iota(S_i))$ has a
nonempty zero-sum subsequence; let $\emptyset \neq I \subset
[1,3m-2]$ be such that $\sum_{i \in I} \sigma(\iota(S_i))=0$. Thus
$\sigma(\iota(\prod_{i \in I}S_i))= \sigma(\prod_{i \in I}S_i) +
|\prod_{i \in I}S_i|e = 0$, whence $\prod_{i\in I}S_i$ is a nonempty
zero-sum subsequence of $S$ of length divisible by $\ord(e)=n$.
\end{proof}

Parts of the proof of Theorem \ref{ind_thm_equ} are similar to the proof of Theorem \ref{ind_thm_bound}.

\medskip
\begin{proof}[Proof of Theorem \ref{ind_thm_equ}] Let $m=\mathsf D(G_q)-\exp(G_q)+1$.
Our assumptions on $G$ imply that there exists some $q\in \P$ and
$q$-group $H$ such that $G \cong H \oplus C_n$ with $\exp(H)\mid n$.
Moreover, we know that $$\mathsf{D}(H)=m$$ divides $\exp(G_q)$, and thus $n$ as well; let $n = mk$.
Let $K\cong C_k$ be a subgroup of $G$ such that $G/K \cong H \oplus
C_m$. Let $\varphi \colon G \to G/K$ denote the canonical
epimorphism. Since $m$ divides $\exp(G_q)$, which is a power of the prime $q$, it follows that $m$ is itself a power of $q$. Consequently, since $\exp(H)\leq \mathsf D^*(H)\leq \mathsf D(H)=m$ with $\exp(H)$ also a power of the prime $q$, it follows that \be\label{weesel}\exp(H)\mid m.\ee Since $H$ and $G_q$ are both $q$-groups, so that $\mathsf D(H)=\mathsf D^*(H)$ and $\mathsf D(G_q)=\mathsf D^*(G_q)$ (as remarked earlier in the paper), it follows that \be\label{twister}\mathsf D(G_q)-\exp(G_q)=\mathsf D (H)-1.\ee

We start by establishing the result on $\eta(G)$ and
$\mathsf{s}(G)$. On the one hand, by \cite[Lemma 3.2]{E-E-G-K-R07} and \eqref{twister},
we know \ber \nn\eta(G) \ge 2 ( \mathsf{D}(H) - 1 ) + n=2(\mathsf D(G_q)-\exp(G_q))+n\und\\\nn
\mathsf{s}(G) \ge 2 ( \mathsf{D}(H) - 1 ) + 2 n - 1=2(\mathsf D(G_q)-\exp(G_q))+2n-1.\eer For the upper bound, first observe that
\eqref{weesel} implies that $\exp(H\oplus C_m)=m$. In consequence, we have $\exp(H\oplus C_m)\exp(C_k)=mk=n=\exp(G)$. Thus \eqref{inductive-exp-bounds-eta} and \eqref{inductive-exp-bounds-s} imply that
\be\label{waxing}\eta(G) \le m (\eta(K) - 1) + \eta(H \oplus C_m)\und \mathsf{s}(G) \le m (\mathsf{s}(K) - 1) + \mathsf{s}( H \oplus
C_m ).\ee
Since $K\cong C_k$ is cyclic, we know (see \cite[Theorem 5.8.3]{Ge-HK06a}) \be\label{tree}\eta(K)=k\und \mathsf s(K)=2k-1.\ee Noting that $H\oplus C_m$ is a $q$-group with $q$ an odd prime so that \eqref{weesel} implies $$\mathsf D(H\oplus C_m)=\mathsf D(H)+m-1=2m-1=2\exp(H\oplus C_m)-1,$$ we see that we can apply Theorem \ref{ind_thm_bound} to conclude \be\label{eert} \eta(H \oplus C_m)\leq 3m-2\und \mathsf{s}(H \oplus C_m) \leq 4m-3.\ee Combing \eqref{waxing}, \eqref{tree} and \eqref{eert} yields $$\eta(G)\leq 2m-2+mk=2(\mathsf D(H)-1)+n\und \mathsf s(G)\leq 2m-2+2mk-1=2(\mathsf D(H)-1)+2n-1,$$ as desired.

It remains to determine $\zsm{d}{G}$. We continue to use the notation already introduced. By hypothesis, we have $m|d$; as shown above, we also have $G\cong H\oplus C_n$ with $\exp(H)|m$ and $m|n$. Thus it follows, in view of \eqref{twister} and $\mathsf D(H)=\mathsf D^*(H)$ (as $H$ is a $q$-group), that
\ber\nn\mathsf{D}^{\ast}(G\oplus C_d)&=& \mathsf{D}^{\ast}(H) + \gcd(n,d) + \lcm(n,d) - 2\\
&=&\mathsf{D}(G_q)-\exp(G_q) + \gcd(n,d) + \lcm(n,d) - 1.\nn\eer
By Proposition \ref{onsd_prop_bounds}, we know the above  quantity is a lower bound for
$\zsm{d}{G}$. It remains to show it is also an upper bound as well.

Let $S
\in \mathcal{F}(G)$ be of the above length $\mathsf D^*(G\oplus C_d)=\mathsf{D}^{\ast}(H) + \gcd(n,d) + \lcm(n,d) - 2$. As used in the proof for the bounds $\eta(G)$ and $\mathsf s(G)$, we know that $\exp(H\oplus C_m)=m$ and
\be\label{dex}\mathsf{s}(H\oplus C_m) \leq  4m - 3\ee by Theorem \ref{ind_thm_bound}. We set $j=\gcd(n,d)/m +
\lcm(n,d)/m - 2$. Then, recalling that $\mathsf D^*(H)=\mathsf D(H)=m$, we find that  \be\label{oaktreewalz}|S|=\mathsf{D}(H) +
\gcd(n,d) + \lcm(n,d)-2 = m ( j-1 ) + 4m -2.\ee As a result, repeating applying, in view of \eqref{dex}, the definition of $\mathsf{s}(H\oplus C_m)$ to $\varphi(S)$ and recalling that $\exp(H\oplus C_m)=m$ in view of \eqref{weesel},  it follows that $S$ admits a
product decomposition $S = S_1 \cdot \ldots \cdot S_j S'$ such that each $\varphi (S_i)$ has sum zero and length  $|S_i|=m$, where
$S_1, \ldots, S_j, S' \in \mathcal F (G)$  (see
\cite[Lemma 5.7.10]{Ge-HK06a}). Since \eqref{oaktreewalz} implies $$|S'| =|S|-jm= m ( j-1 ) + 4m -2-jm=3m -2$$ and since $\zsm{m}{H
\oplus C_m} \leq  3m - 2$ by Theorem \ref{ind_thm_bound}, which we can invoke as explained before \eqref{eert},  it follows that $S'$ has a
subsequence $S_{j+1}$ with $m \mid |S_{j+1}|$ and
$\sigma(S_{j+1})\in K$.

We consider $\iota \colon G \to G \oplus
C_d$ defined via $\iota(g)=g+e$ for some generating element $e$ of
$C_d$. We observe that $\sigma(\iota(S_i))\in K \oplus \langle m e
\rangle$ for each $i \in [1,j+1]$. Since $m \mid d$ and $n=mk$, it follows that
$$K \oplus \langle me_i\rangle \cong C_{n/m} \oplus C_{d/m} \cong
C_{\gcd(n,d)/m}\oplus C_{\lcm(n,d)/m}.$$ This is a rank $2$ group, so the Davenport constant
of this group is $j+1$ cf.~the results mentioned before the proof
of Theorem \ref{onsd_thm}. Hence the sequence
$\prod_{i=1}^{j+1}\sigma(\iota(S_i))\in \mathcal F(K \oplus \langle me_i\rangle)$ has a nonempty zero-sum
subsequence. Let $\emptyset \neq I \subset [1,j+1]$ denote
index-set corresponding to this sequence. It follows that
$\prod_{i\in I}\iota(S_i)\in \mathcal F(G\oplus C_d)$ is a zero-sum sequence,   whence
$\prod_{i\in I}S_i\in\mathcal F(G)$ is a zero-sum subsequence of $S$ with length divisible by
$d$ (by the same arguments used at the end of the proof of Theorem \ref{ind_thm_bound}).
\end{proof}

We end this section by discussing the relevance of the assumptions in our results.
\begin{remark} \
\label{ind_rem_imp}
\begin{enumerate}
\item It is conceivable that the assumption  $\mathsf{D}(G_q) - \exp(G_q)+1 \mid \exp(G_q)$ in Theorem \ref{ind_thm_equ}
 can actually be replaced by the assumption $\mathsf{D}(G_q) - \exp(G_q)+1 \le \exp(G_q)$ of Theorem \ref{ind_thm_bound}.
We could relax the assumption in this way if \cite[Conjecture 4.1]{Sc-Zh09a} were true; this conjecture concerns the exact value
of $\eta(G_q)$ and $\mathsf{s}(G_q)$ under this weaker assumption.
\item The restriction that $\exp(G)$ and $q$ are odd, which is imposed in the second part of our result, is due to the fact that \cite[Theorem 1.2]{Sc-Zh09a} is only applicable in this case, yet it is well possible that the statement holds for $2$-groups as well, in which case these assumptions could be dropped (cf.~again \cite[Conjecture 4.1]{Sc-Zh09a}).
\item The restriction in Theorem \ref{ind_thm_equ} that $G/G_q$  is cyclic is very likely not technical.
It seems quite unlikely that there is a uniform argument of this form to determine the precise value of the constants under the assumptions of Theorem \ref{ind_thm_bound}. For example, note that in this more general setting, $\mathsf{D}^{\ast}(G)$ depends on the \emph{precise} structure of each of the $p$-subgroups of $G$ (also see the results in \cite{E-E-G-K-R07}).
Yet, imposing the assumption that $G$ is a group of rank $2$, and thus each $p$-subgroup has at most rank $2$, the values of $\mathsf{D}(G)$, $\eta(G)$, and $\mathsf{s}(G)$ are known, and we additionally determine $\zsm{d}{G}$ for
any $d$ (see Section \ref{sec-rank2}).
\end{enumerate}
\end{remark}

\bigskip
\section{On $\zsm{d}{G}$ for groups of rank two}\label{sec-rank2}
\bigskip

In this section, we determine $\zsm{d}{G}$ for rank $2$ groups $G$. For the proof, we make use of the fact that \be\label{kemnitz}\mathsf s(C_m\oplus C_n)=2n+2m-3\ee when $1\leq m|n$ (see \cite[Theorem 5.8.3]{Ge-HK06a}), which is essentially a consequence of the Kemnitz Conjecture, verified by Reiher \cite{Re07a}. We also need the fact that \be\label{cyclops}\mathsf D(C_m\oplus C_n)=m+n-1\ee when $1\leq m\mid n$ (see \cite[Theorem 5.8.3]{Ge-HK06a}).

\bigskip
\begin{lemma}\label{useful-lemma} Let  $G\cong C_m\oplus C_n$ with $1\leq m\mid n$, and let $t\in \N$.  If $S\in \mathcal F(G)$ is a sequence with $$|S|\geq (t-1)n+\zsm{n}{G},$$ then $S$ has a decomposition $S=S_1\cdot\ldots\cdot S_tS'$ with each $S_i$ zero-sum, $|S_i|=n$ for $i\in [1,t-1]$, and $|S_t|\in \{n,2n\}$, where $S_1,\ldots,S_n,S'\in \mathcal F(G)$.
\end{lemma}

\begin{proof}
In view of Theorem \ref{ind_thm_bound}.1(b), we know that $|S''|\geq \zsm{n}{G}$ implies that $S''\in \mathcal F(G)$
contains a zero-sum sequence of length $n$ or $2n$.
From Proposition \ref{onsd_prop_bounds} and Lemma \ref{D*lemma}, we know \be\label{willful}\zsm{n}{G}\geq \mathsf D^*(G\oplus C_n)=n+\mathsf D^*(G)-1=2n+m-2.\ee From \eqref{kemnitz}, we know \be\label{skillful} \mathsf s(G)\leq 2n+2m-3.\ee

In view of \eqref{willful} and \eqref{skillful}, we have $n+\zsm{n}{G}\geq 3n+m-2\geq \mathsf s(G)$. Thus, in view of $|S|\geq (t-1)n+ \zsm{n}{G}$ , we can repeatedly apply the definition of $\mathsf s(G)$ to $S$ to find $t-1$ zero-sum subsequences $S_1,\ldots,S_{t-1}$  with $S_1\cdot\ldots\cdot S_{t-1}\mid S$ and $|S_i|=n$ for all $i$. Let $S''=S(S_1\cdot\ldots\cdot S_{t-1})^{-1}$. Then $|S''|=|S|-(t-1)n\geq  \zsm{n}{G}$. Hence, as remarked at the beginning of the proof, $S''$ must have a zero-sum subsequence $S_t$ with $|S_t|=n$ or $|S_t|=2n$, completing the proof.
\end{proof}

\bigskip

\begin{theorem}\label{thm-rank2} Let $d\in \mathbb N$ and let $G\cong C_m\oplus C_n$ with $1\leq m\mid n$. Then $$\zsm{d}{G}=\mathsf D^*(G\oplus C_d)=\lcm(n,d)+\gcd(n,\lcm(m,d))+\gcd(m,d)-2.$$
\end{theorem}

\begin{proof} When $m=1$, this follows from Theorem \ref{onsd_thm}.1. Therefore we assume $m>1$.
Since $G$ has rank two, it follows that each $p$-component $G_p$ has rank at most two, and thus $\mathsf D(G_p)=\mathsf D^*(G_p)\leq 2\exp(G_p)-1$ for all primes dividing $n$.
Note that Lemma \ref{D*lemma} implies that \be\label{target-value}\mathsf D^*(G\oplus C_d)=\lcm(n,d)+\gcd(n,\lcm(m,d))+\gcd(m,d)-2,\ee while Proposition \ref{onsd_prop_bounds} shows that this is a lower bound for $\zsm{d}{G}$. It remains to show it is also an upper bound. We begin by considering two particular cases.

\subsection*{Case 1: } $d=n$. If $m=n$, then \eqref{target-value} becomes $\mathsf D^*(G\oplus C_d)=3n-2$, and the result follows from Theorem \ref{ind_thm_bound}.1(b). Therefore we assume $m<n$. We proceed by a minor modification of the argument used for
 Theorem \ref{ind_thm_bound}.1(b).
Since $m<n$, let $n=km$ with $k\geq 2$. Let $K\leq G$ be a subgroup such that $$K\cong C_k\und G/K\cong C_m\oplus C_m$$ and let $\varphi:G\rightarrow G/K\cong C_m\oplus C_{m}$ denote the natural homomorphism. Note, under the assumption $d=n$, that \eqref{target-value} becomes $$\mathsf D^*(G\oplus C_d)=2n+m-2.$$ Let $S\in \mathcal F(G)$ with $|S|=2n+m-2$. By the previously handled case ($d=m=n$), it follows that $\zsm{m}{G/K}=\mathsf D^*(G/K\oplus C_{m})=3m-2$.
Thus $$|\varphi(S)|=|S|=(2k-2)m+3m-2=(2k-2)m+\zsm{m}{G/K}.$$ Applying Lemma \ref{useful-lemma} to $\varphi(S)$, we find a product decomposition $S=S_1\cdot\ldots\cdot S_{2k-1}S'$ with each $S_i$ being zero-sum modulo $K$ and of length $|S_i|\in \{m,2m\}$. Let $\iota:G\rightarrow G\oplus \langle e\rangle\cong G\oplus C_n$, where $\ord(e)=n$, be the map defined by letting $\iota(g)=g+e$. Then, since each $S_i$ is zero-sum modulo $K$ with length a multiple of $m$, it follows that  $\sigma(\iota(S_i))\in K\oplus \langle me\rangle\cong C_k\oplus C_k$ for each $i\in [1,2k-1]$. Since $\mathsf D(C_k\oplus C_k)=2k-1$ by \eqref{cyclops}, applying the definition of $\mathsf D(C_k\oplus C_k)$ to the sequence $\prod_{i=1}^{2k-1}\sigma(\iota(S_i))\in \mathcal F(K\oplus \langle me\rangle)$ yields a nonempty zero-sum sequence, say indexed by $I\subseteq [1,2k-1]$. Thus $0=\sigma(\prod_{i\in I} \iota(S_i))=\sigma(\prod_{i\in I}S_i)+|\prod_{i\in I}S_i|e$, whence $\prod_{i\in I}S_i\in \mathcal F(G)$ is a nonempty zero-sum subsequence of $S$ whose length is divisible by  $\ord(e)=n$, as desired. This completes the case $d=n$.

\subsection*{Case 2:} $d\mid n$. Let $u=\frac{\lcm(m,d)}{d}=\frac{m}{\gcd(m,d)}$ and  $v=\frac{n}{\lcm(m,d)}$. Note $uvd=n$. Let $K\leq G$ be a subgroup such that $$K\cong  C_u\oplus C_{uv}\und G/K\cong C_{\gcd(m,d)}\oplus C_d$$ and let $\varphi:G\rightarrow G/K\cong C_{\gcd(m,d)}\oplus C_d$ denote the natural homomorphism.
Note, under the assumption $d\mid n$, that \eqref{target-value} becomes \be\label{gasp}\mathsf D^*(G\oplus C_d)=n+\lcm(m,d)+\gcd(m,d)-2.\ee
Let $S\in \mathcal F(G)$ be a sequence with $$|S|=n+\lcm(m,d)+\gcd(m,d)-2=(uv+u-2)d+2d+\gcd(m,d)-2.$$ In view of Case 1 and \eqref{target-value}, we have $\zsm{d}{G/K}=\mathsf D^*(G/K\oplus C_d)=2d+\gcd(m,d)-2$. Thus, applying Lemma \ref{useful-lemma} to $\varphi(S)$, we find a product decomposition $S=S_1\cdot\ldots\cdot S_{uv+u-1}S'$ with each $S_i$ zero-sum modulo $K$ and of length divisible by $d$. But now, in view of \eqref{cyclops}, the sequence $\prod_{i=1}^{uv+u-1}\sigma(S_i)\in \mathcal F(K)$ has length $uv+u-1=\mathsf D(C_u\oplus C_{uv})=\mathsf D(K)$. Hence, applying the definition of $\mathsf D(K)$ to $\prod_{i=1}^{uv+u-1}\sigma(S_i)$, we find a non-empty zero-sum subsequence, say indexed by $I\subseteq [1,uv+u-1]$. Thus $\sigma(\prod_{i\in I}S_i)=0$. Moreover, since $d\mid |S_i|$ for each $i$, it follows that $d\mid |\prod_{i\in I}S_i|$, as desired. This completes the case $d\mid n$.

\bigskip

We now proceed to show \be\label{super-bound}\zsm{d}{G}\leq \lcm(n,d)-n+\zsm{\gcd(n,d)}{G}.\ee Once \eqref{super-bound} is established, then, applying Case 2 to $\zsm{\gcd(n,d)}{G}$ and using \eqref{gasp}, we will know   \ber \nn \zsm{d}{G}&\leq& \lcm(n,d)-n+\mathsf D^*(G\oplus C_{\gcd(n,d)})\\\nn &=& \lcm(n,d)-n+(n+\lcm(m,\gcd(n,d))+\gcd(m,\gcd(n,d))-2)\\\nn &=&
\lcm(n,d)+\lcm(m,\gcd(n,d))+\gcd(m,d)-2,
\eer which is equal to $\mathsf D^*(G\oplus C_d)$ by Lemma \ref{D*lemma}. In consequence, once \eqref{super-bound} is established, the proof will be complete. We continue with the proof of \eqref{super-bound}. As \eqref{super-bound} holds trivially when $d\mid n$, we assume $d\nmid  n$.

Let $\alpha n=\lcm(n,d)$. Then, since $d\nmid n$, we have $\alpha\geq 2$. Let $S\in \mathcal F(G)$ be a sequence with \be\label{butterflys}|S|=\lcm(n,d)-n+\zsm{\gcd(n,d)}{G}=(\alpha-1)n+\zsm{\gcd(n,d)}{G}.\ee By Case 2 and \eqref{target-value}, we have \be\label{usetee}\zsm{\gcd(n,d)}{G}=n+\lcm(m,\gcd(n,d))+\gcd(m,\gcd(n,d))-2\geq n+m-1.\ee Thus it follows  from \eqref{kemnitz} that  $$2n+\zsm{\gcd(n,d)}{G}\geq 3n+m-1\geq \mathsf s(G).$$ Consequently, in view of \eqref{butterflys} and $\alpha\geq 2$, it follows, by repeatedly applying the definition of $\mathsf s(G)$ to $S$, that we can find $\alpha-2$ zero-sum subsequences $S_1,\ldots,S_{\alpha -2}\in \mathcal F(G)$ such that $S_1\cdot\ldots\cdot S_{\alpha-2}|S$ and $|S_i|=n$ for all $i\in [1,\alpha-2]$. Let $S'=S(S_1\cdot\ldots\cdot S_{\alpha-2})^{-1}$. Then, in view of  \eqref{usetee} and Case 1, we have $$|S'|=|S|-(\alpha-2)n=n+\zsm{\gcd(n,d)}{G}\geq 2n+m-1\geq \zsm{n}{G}.$$ Hence, since $\zsm{n}{G}<3n$, applying the definition of $\zsm{n}{G}$ to $S'$ yields a zero-sum subsequence $S_{\alpha-1}\mid S'$ with $|S_{\alpha-1}|=n$ or $|S_{\alpha-1}|=2n$.
If $|S_{\alpha-1}|=2n$, then $S_1\cdot\ldots\cdot S_{\alpha-1}$ is a zero-sum subsequence of $S$ with length $(\alpha-2)n+2n=\alpha n=\lcm(n,d)$, which is a multiple of $d$, and thus of the desired length. Therefore we may instead assume $|S_{\alpha-1}|=n$.
Let $S''=S(S_1\cdot\ldots\cdot S_{\alpha-1})^{-1}$. Then $|S''|=|S|-(\alpha-1)n=\zsm{\gcd(n,d)}{G}$, so applying the definition of $\zsm{\gcd(n,d)}{G}$ to $S''$ yields a zero-sum sequence $S_0\mid S''$ with length $|S_0|=k_0\gcd(n,d)$ for some $k_0\in \N$.

Since $\alpha n=\lcm(n,d)$, it follows that $$d=\alpha\gcd(n,d).$$ Let $n=n'\gcd(n,d)$. Then, since $d=\alpha\gcd(n,d)$, we see that $$\gcd(\alpha,n')=1.$$ If $k_0\equiv 0\mod \alpha$, then $$|S_0|=k_0\gcd(n,d)\equiv \alpha\gcd(n,d)\equiv 0\mod \alpha\gcd(n,d),$$ in which case, since $\alpha\gcd(n,d)=d$, we see that $S_0$ is a zero-sum subsequence of length divisible by $d$, as desired. Therefore we may assume $k_0\not\equiv 0\mod \alpha$.

Observe that $$|S_1\cdot\ldots\cdot S_j|=jn=jn'\gcd(n,d)\quad\mbox{ for }j\in [1,\alpha-1].$$ Thus, since $\gcd(\alpha,n')=1$, we conclude that
$$\{\frac{1}{\gcd(n,d)}|S_1|,\frac{1}{\gcd(n,d)}|S_1S_2|,\ldots,\frac{1}{\gcd(n,d)}|S_1\cdot\ldots\cdot S_{\alpha-1}|\}$$ is a full set of nonzero residue classes modulo $\alpha$. Consequently, since $k_0\not\equiv 0\mod \alpha$, we can find $k\in [1,\alpha-1]$ such that $\frac{1}{\gcd(n,d)}|S_1\cdot\ldots\cdot S_k|+k_0\equiv 0\mod \alpha$, in which case $$|S_0S_1\cdot\ldots\cdot S_k|=|S_1\cdot\ldots\cdot S_k|+k_0\gcd(n,d)\equiv 0\mod \alpha\gcd(n,d).$$ Since $\alpha\gcd(n,d)=d$, this means that $S_0S_1\cdot\ldots\cdot S_k$ is a subsequence of $S$ with length divisible by $d$. Moreover, since each $S_i$ is zero-sum, it follows that the subsequence $S_0S_1\cdot\ldots\cdot S_k$ is also zero-sum, whence we have found a zero-sum sequence of the desired length, completing the proof of \eqref{super-bound}, which completes the proof as remarked earlier.
\end{proof}

\bigskip
\section{Upper bounds  for the lengths of zero-sum subsequences} \label{4}
\bigskip

Let $H$ be a Krull monoid with class group $G$ and suppose that
every class contains a prime divisor. The investigation of sets of
lengths of the form $\mathsf L (uv)$, where $u,v \in H$ are
irreducible elements, is a frequently studied topic in the theory of
non-unique factorizations (see for example \cite[Section
6.6]{Ge-HK06a}). Only recently, a close connection of this topic
with the catenary degree $\mathsf c (H)$ of $H$ was found---see
the invariant $\daleth (H)$ introduced in \cite{Ge-Gr-Sc11a}. As
is well-known, the study of sets $\mathsf L (uv)$ translates into a
zero-sum problem as follows: pick two minimal zero-sum sequences $U$
and $V$ over $G$ and find product decompositions of the form $UV = W_1 \cdot
\ldots \cdot W_k$ with  $W_1, \ldots, W_k$  minimal zero-sum
sequences over $G$. To control the number $k$ of atoms in such a factorization, it is
desirable to be able to find zero-sum subsequences of the (long) zero-sum sequence $UV$ with
bounded lengths (see Condition
$(b)$ in Lemma \ref{4.1}).

Thus, in zero-sum terminology,  we have to study conditions which
imply that, for a given $d \in [1, \mathsf D (G)-1]$,  every zero-sum
sequence $A \in \mathcal F (G)$ of length $|A| \ge \mathsf D (G)+1$
has a zero-sum subsequence $T$ of length $|T| \in [1, d]$. Since, by
definition, $\mathsf D (G)$ is the maximal length of a minimal
zero-sum sequence, it makes no sense to consider the above question
for sequences $A$ of length less than $\mathsf{D}(G)+1$. We start
with a simple characterization of this property which allows us to
obtain a natural restriction for $d$.

\medskip
\begin{lemma} \label{4.1}
Let  $d \in \N$ with $\mathsf D (G) \le 2d-1$. Then the following
statements are equivalent\,{\rm :}
\begin{enumerate}
\item[(a)] For all $U, V \in \mathcal A (G)$ with $|UV| \ge 2d$, there exists
           a zero-sum subsequence $T$ of $UV$ of length $|T| \in [1, d]$.

\item[(b)] For all $U, V \in \mathcal A (G)$,  there exists
           a zero-sum subsequence $T$ of $UV$ of length $|T| \in [1, d]$.

\item[(c)] Every zero-sum sequence $A \in \mathcal F (G)$ of length $|A|
           \ge \mathsf D (G)+1$ has
           a zero-sum subsequence $T$  of length $|T| \in [1, d]$.
\end{enumerate}
\end{lemma}

\begin{proof}
(a) \,$\Rightarrow$\, (b) \ Let $U, V \in \mathcal A (G)$ be given,
say $|U| \le |V|$. If $|UV| \ge 2d$, then the assertion follows from
(a). If $|UV| \le 2d$, then we set $T = U$ and get $2|T| \le |UV|
\le 2d$.

\smallskip
\noindent
(b) \,$\Rightarrow$\, (c) \ Let $A \in \mathcal F (G)$ be zero-sum. Then there
are  $U_1, \ldots, U_k \in \mathcal A (G)$ such that
$A = U_1 \cdot\ldots\cdot U_k$. Since $|A|  \ge \mathsf D (G)+1$,
it follows that $k \ge 2$. Thus there exists a zero-sum sequence $T$
with $T \mid U_1 U_2$, and hence with $T \mid A$ also, such that $|T| \in [1, d]$.

\smallskip
\noindent
(c) \,$\Rightarrow$\, (a) \ Obvious.
\end{proof}

\medskip
\begin{remark} \label{4.2}
Let $d \in \N$. In general, none of the statements in the previous
lemma can hold without the assumption $\mathsf D (G) \le 2d-1$. This
can be seen from the following example. Take $G = H \oplus H$ such
that $\mathsf D (G) = 2 \mathsf D (H) - 1$ (note that this holds
true if $H$ is cyclic or a $p$-group). Then there are $U, V \in
\mathcal A (G)$ such that $\langle \supp (U) \rangle \cap \langle
\supp (V) \rangle = \{0\}$ and $|U| = |V| = \mathsf D (H)$. Thus the
only nonempty zero-sum subsequences of $UV$ are $U$ and $V$, which
have length
\[
|U| = |V| = \frac{\mathsf D (G) + 1}{2} \,.
\]
\end{remark}

We give the main result of this section; see below for groups fulfilling the assumptions.

\medskip
\begin{theorem} \label{4.3}
Let  $d \in \N$ with $\mathsf D (G) \le 2d-1$ and suppose that
$\mathsf s_{d\N} (G) \le 3d-1$.
\begin{enumerate}
\item Every sequence $S \in \mathcal F (G)$ of length $|S| = \mathsf s_{d\N}
      (G)$ has a zero-sum subsequence $T$ of length $|T| \in [1,d]$.
\item Every zero-sum sequence $A \in \mathcal F (G)$ of length $|A|
      \ge \mathsf D (G)+1$ has a zero-sum subsequence $T$ of length $|T| \in [1,d]$.
\end{enumerate}
\end{theorem}

\begin{proof}
1. Let $S \in \mathcal F (G)$ be a sequence of length $|S| = \mathsf s_{d\N} (G)$.
 Since
$\mathsf s_{d\N} (G) \le 3d-1$, $S$ has a zero-sum subsequence $T$
of length $|T| \in \{d, 2d \}$. If $|T| = d$, then we are done. If
$|T| = 2d$, then $2d \ge \mathsf D (G)+1$ implies that $T$ has a
product decomposition $T = T_1 T_2$ with $T_1$ and $T_2$ nonempty
zero-sum sequences. Clearly, we have $\min \{|T_1|, |T_2| \} \in
[1,d]$.

\smallskip
\noindent
2. Let $A \in \mathcal F (G)$ be zero-sum with $|A|
      \ge \mathsf D (G)+1$. Then $A$ is a product of two nonempty
      zero-sum subsequences, and if $|A| \le 2d$, then the assertion
      is clear as before. Suppose that
$|A| \ge 2d+1$. If $|A| \ge \mathsf s_{d\N} (G)$, then the assertion
follows from 1. Therefore we have \be\label{bearpolka}2d+1 \le |A| <  \mathsf s_{d \N} (G) \le
3d-1.\ee Now the sequence
\[
S = 0^kA, \quad \text{where} \quad k = \mathsf s_{d \N} (G) - |A| \in
[1, d-2],
\]
has a zero-sum subsequence $T = 0^{k'} A'$ of length $|T| \in \{d,
2d\}$, where $k' \in [0,k]$ and $A' \mid A$. If $|T| = d$, then $A'$ is
a  zero-sum subsequence of $A$ of length $|A'| \in [1,d]$, as desired. If $|T| =
2d$, then $A'$ is a zero-sum subsequence of length $$|A'|\geq 2d-k=2d+|A|-\mathsf s_{d \N} (G).$$ Hence,
${A'}^{-1}A$ is a zero-sum subsequence (as both $A$ and $A'$ are zero-sum sequences) with length (in view of \eqref{bearpolka})
$$|{A'}^{-1}A|=|A|-|A'|\leq |A|-(2d+|A|-\mathsf s_{d \N} (G))= \zsm{d}{G}-2d\leq 3d-1-2d=d-1.$$ Moreover, since \eqref{bearpolka} implies $|A|\geq 2d+1$ while $A'\mid T$ implies $|A'|\leq |T|=2d$, we see that ${A'}^{-1}A$ is also a nonempty zero-sum subsequence, and the proof is complete in this case as well.
\end{proof}

\medskip

Results of the two preceding sections yield various classes of groups fulfilling the conditions
of Theorem \ref{4.3}.
The groups covered by the assumptions of Theorem \ref{ind_thm_bound}.1, thus in particular groups of rank two, fulfil the conditions of Corollary \ref{4.4}. In the special case of groups of rank two, the result was first
achieved in \cite[Lemma 3.6]{Ge-Gr09c}.

\medskip
\begin{corollary} \label{4.4}
Let  $\exp (G) = n$ and suppose that $\mathsf D (G) \le 2n-1$ and $\mathsf{s}_{n\N} (G) \le 3n-1$.
Then every zero-sum sequence $A \in \mathcal{F}(G)$ of length $|A| \ge \mathsf{D}(G)+1$ has a nonempty  zero-sum subsequence of length at most $\exp(G)$.
\end{corollary}

\begin{proof}
This is a special case of Theorem \ref{4.3}.2.
\end{proof}

\medskip
\begin{corollary} \label{4.5}
Let $G$ be a $p$-group. Suppose there exists some $i \in [1, \mathsf{D}(G)]$ such that $( \mathsf{D}^{\ast} (G)+i)/2$ is a power of $p$. Then
every zero-sum sequence $A \in \mathcal F (G)$ of length $|A|\ge \mathsf{D}(G)+1$ has a zero-sum subsequence $T$ of length $|T| \in [1, (\mathsf{D}^{\ast} (G)+i)/2 ]$.
\end{corollary}

\begin{proof}
We set $d = ( \mathsf{D}^{\ast} (G)+i)/2$. Then $2d = \mathsf{D}^{\ast} (G) +
i \ge \mathsf{D} (G)+1$, and thus Theorem \ref{onsd_thm}.2(a) implies that
$\zsm{d}{G} \le \mathsf D^*(G)+d-1\leq \mathsf D(G)+d-1\leq 3d-2$. Therefore the assertion follows
from Theorem \ref{4.3}.
\end{proof}

\medskip
Note, if $(\mathsf{D}^{\ast}(G)+1)/2$ is a power of $p$, then the
above result is best possible, as can be seen from the example
discussed in Remark \ref{4.2}.

% \bibliography{ger,hk,fact,zerosum,ideal}

\begin{thebibliography}{10}

\bibitem{AGS}
S.D. Adhikari, D.J. Grynkiewicz, and Z.-W. Sun,
\emph{On weighted zero-sum sequences}, manuscript.


\bibitem{Bh-SP07a}
G.~Bhowmik and J.-C. Schlage-Puchta, \emph{Davenport's constant for
groups of
  the form $\mathbb{Z}_3 \oplus \mathbb{Z}_3 \oplus \mathbb{Z}_{3d}$}, Additive
  Combinatorics (A.~Granville, M.B. Nathanson, and J.~Solymosi, eds.), CRM
  Proceedings and Lecture Notes, vol.~43, American Mathematical Society, 2007,
  pp.~307 -- 326.

\bibitem{Restricted4}
R.~Chi, S.~Ding, W.~Gao, A.~Geroldinger, and W.A.~Schmid,
\emph{On zero-sum subsequences of restricted size {IV}},
Acta Math. Hungar. \textbf{107} (2005), 337 -- 344.

\bibitem{Delorme}
C.~Delorme, O.~Ordaz, and D.~Quiroz, \emph{Some remarks on {Davenport constant}}, Discrete Math. \textbf{237} (2001), 119 -- 128.

\bibitem{E-E-G-K-R07}
Y.~Edel, C.~Elsholtz, A.~Geroldinger, S.~Kubertin, and L.~Rackham,
  \emph{Zero-sum problems in finite abelian groups and affine caps}, Quarterly.
  J. Math., Oxford II. Ser. \textbf{58} (2007), 159 -- 186.

\bibitem{Fr-Sc10a}
M.~Freeze and W.A. Schmid, \emph{Remarks on a generalization of the
{D}avenport
  constant}, manuscript.

\bibitem{Restricted3}
W.~Gao, \emph{On zero sum subsequences of restricted size {III}},
Ars Combin. \textbf{61} (2001), 65 -- 72.

\bibitem{Ga-Ge06b}
W.~Gao and A.~Geroldinger, \emph{Zero-sum problems in finite abelian
groups{\rm
  \,:} a survey}, Expo. Math. \textbf{24} (2006), 337 -- 369.

\bibitem{Ga-Ge07a}
\bysame, \emph{On the number of subsequences with given sum of
sequences over
  finite abelian $p$-groups}, Rocky Mt. J. Math. \textbf{37} (2007), 1541 --
  1550.

\bibitem{Ga-Ge-Sc07a}
W.~Gao, A.~Geroldinger, and W.A. Schmid, \emph{Inverse zero-sum
problems}, Acta
  Arith. \textbf{128} (2007), 245 -- 279.

\bibitem{Ga-Ha-Wa10a}
W.~Gao, Y.~ould Hamidoune, and G.~Wang, \emph{Distinct lengths
modular zero-sum
  subsequences: a proof of {G}raham's conjecture}, manuscript.

\bibitem{Ga-Pe09a}
W.~Gao and J.~Peng, \emph{On the number of zero-sum subsequences of
restricted
  size}, Integers \textbf{9} (2009), Paper A41, 537 -- 554.

\bibitem{GaoT}
W.~Gao and R.~Thangadurai, \emph{On zero-sum sequences of prescribed length},
Aequationes Math. \textbf{72} (2006), 201 -- 212.

\bibitem{GeKL}
A.~Geroldinger, \emph{On a conjecture of {K}leitman and {L}emke},
J. Number Theory \textbf{44} (1993), 60 -- 65.

\bibitem{Ge09a}
A.~Geroldinger, \emph{Additive group theory and non-unique
factorizations},
  Combinatorial {N}umber {T}heory and {A}dditive {G}roup {T}heory
  (A.~Geroldinger and I.~Ruzsa, eds.), Advanced Courses in Mathematics CRM
  Barcelona, Birkh{\"a}user, 2009, pp.~1 -- 86.

\bibitem{Ge-Gr09c}
A.~Geroldinger and D.J. Grynkiewicz, \emph{On the structure of
minimal zero-sum
  sequences with maximal cross number}, J. Combinatorics and Number Theory
  \textbf{1 (2)} (2009), 9 -- 26.

\bibitem{Ge-Gr-Sc11a}
A.~Geroldinger, D.J. Grynkiewicz, and W.A. Schmid, \emph{The
catenary degree of
  {K}rull monoids {I}}, manuscript.

\bibitem{Ge-HK06a}
A.~Geroldinger and F.~Halter-Koch, \emph{Non-{U}nique
{F}actorizations.
  {A}lgebraic, {C}ombinatorial and {A}nalytic {T}heory}, Pure and Applied
  Mathematics, vol. 278, Chapman \& Hall/CRC, 2006.

\bibitem{Ge-Li-Ph11}
A.~Geroldinger, M.~Liebmann, and A.~Philipp, \emph{On the
{D}avenport constant
  and on the structure of extremal sequences}, manuscript.

\bibitem{Ge-Sc92}
A.~Geroldinger and R.~Schneider, \emph{On {D}avenport's constant},
J. Comb.
  Theory, Ser. A \textbf{61} (1992), 147 -- 152.

\bibitem{Gi12a}
B.~Girard, \emph{On the existence of distinct lengths zero-sum
subsequences},
  Rocky Mt. J. Math., to appear.

\bibitem{Gr10a}
D.J. Grynkiewicz, \emph{Note on a conjecture of {G}raham},
manuscript.

\bibitem{Gr05b}
\bysame, \emph{On a conjecture of {H}amidoune for subsequence sums},
Integers
  \textbf{5(2)} (2005), Paper A07, 11p.

\bibitem{Ka-Pa-Sa10a}
S.S. Kannan and S.K. Pattanayak, \emph{Projective normality of finite
group
  quotients and {EGZ} theorem}.

\bibitem{Ka-Pa-Sa09a}
S.S. Kannan, S.K. Pattanayak, and P.~Sardar, \emph{Projective
normality of
  finite group quotients}, Proc. Am. Math. Soc. \textbf{137} (2009), 863 --
  867.

\bibitem{Kubertin}
S.~Kubertin, \emph{Zero-sums of length {$kq$} in {$\mathbb{Z}^d_q$}},
Acta Arith. \textbf{116} (2005), 145 -- 152.

\bibitem{Re07a}
C.~Reiher, \emph{On {K}emnitz' conjecture concerning lattice points
in the
  plane}, Ramanujan J. \textbf{13} (2007), 333 -- 337.

\bibitem{Sc10c}
W.A. Schmid, \emph{The inverse problem associated to the {D}avenport
constant
  for ${C}_2 \oplus {C}_2 \oplus {C}_{2n} $, and applications to the
  arithmetical characterization of class groups}, submitted.

\bibitem{Sc01}
\bysame, \emph{On zero-sum subsequences in finite abelian groups},
Integers
  \textbf{1} (2001), Paper A01, 8p.

\bibitem{Sc-Zh09a}
W.A. Schmid and J.J. Zhuang, \emph{On short zero-sum subsequences
over
  $p$-groups}, Ars Comb. \textbf{95} (2010).

\bibitem{yuan} P. Yuan, H. Guan, and X. Zeng, \emph{Normal sequences over finite abelian groups}, manuscript.

\end{thebibliography}
% \bibliographystyle{amsplain}

\providecommand{\bysame}{\leavevmode\hbox
to3em{\hrulefill}\thinspace}
\providecommand{\MR}{\relax\ifhmode\unskip\space\fi MR }
\providecommand{\MRhref}[2]{%
  \href{http://www.ams.org/mathscinet-getitem?mr=#1}{#2}
} \providecommand{\href}[2]{#2}

\end{document}